\documentclass[11pt]{article}
\usepackage[utf8]{inputenc}
\usepackage{subfigure}
\usepackage[english]{babel}

\usepackage{graphicx, amsmath, amsthm, amssymb, bm, url}

\usepackage[center]{caption2}
\usepackage{amsfonts,amssymb,amsmath,latexsym,amsthm}
\usepackage{multirow}
\usepackage{cite}

\newtheorem{thm}{Theorem}[section]
\newtheorem{cor}[thm]{Corollary}
\newtheorem{lem}[thm]{Lemma}

\newtheorem{claim}{Claim}

\begin{document}	
	\title{Generalized Tur\'an results for disjoint copies of degenerate graphs}

    \author{%
    Caihong Yang\textsuperscript{\textdagger}\textsuperscript{\textdaggerdbl},%
    \quad Jiasheng Zeng\textsuperscript{\S}%
}
\footnotetext[1]{\scriptsize%
\noindent\textsuperscript{\textdagger}School of Mathematics and Statistics, Fuzhou University, Fuzhou 350108, China;%
\textsuperscript{\textdaggerdbl}Extremal Combinatorics and Probability Group (ECOPRO), Institute for Basic Science (IBS), Daejeon 34126, South Korea. Email: chyang.fzu@gmail.com }
\footnotetext[2]{\scriptsize%
\noindent\textsuperscript{\S}School of Mathematical Sciences, Shanghai Jiao Tong University, Shanghai 200240, China. Email: jasonzeng@sjtu.edu.cn  }
        
	\date{}
	
	\maketitle
	
	\begin{abstract}
    The generalized Tur\'{a}n number $\mathrm{ex}(n, H, F)$ denotes the maximum number of copies of $H$ in an $n$-vertex $F$-free graph. For an integer $t \geq 1$, let $tF$ be the vertex-disjoint union of $t$ copies of $F$. Gerbner, Methuku, and Vizer (2019) established an asymptotically sharp bound for $\mathrm{ex}(n,K_r,(t+1)K_{2,b})$. We extend their results in two directions by considering forbidden graphs $(t+1)K_{a,b}$ and $(t+1)C_{2k}$ and establish more precise matching upper and lower bounds of the same order of magnitude.
     
	\end{abstract}


	
\section{Introduction}

The \emph{Tur\'{a}n number} of a graph $F$, denoted by $\mathrm{ex}(n,F)$, is defined as the maximum number of edges in an $F$-free graph on $n$ vertices. The systematic study of Tur\'{a}n numbers was initiated by Tur\'{a}n~\cite{41Turan} over seventy years ago and has since become a central topic in extremal graph theory (see the surveys~\cite{S13,Keevash11}).

Given two graphs $T$ and $F$, the \emph{generalized Tur\'{a}n number} $\mathrm{ex}(n, T, F)$ is defined as the maximum number of copies of $T$ that can appear in an $F$-free graph on $n$ vertices. This concept was systematically introduced by Alon and Shikhelman~\cite{AS16}. Since Alon and Shikhelman's foundational work, significant attention has been devoted to studying generalized Tur\'{a}n numbers for special structure, as evidenced by recent works~\cite{GGMV20, LA20, BJ24, ZCG23, G2023}. 

Particular attention has been given to counting cliques in $F$-free graphs. Alon and Shihelman~\cite{AS16} proved that $\mathrm{ex}(n,K_r,F)=\Theta(n^{r})$ if and only if the chromatic number $\chi(F)>r$. They also established bounds for the maximum number of cliques in $K_{a,b}$-free graphs, which are sharp for $r\geq 2$, $a\geq 2r-2$, and $b\geq (a-1)!+1$. Moreover, they determined the graphs $F$ for which $\mathrm{ex}(n,K_3,F)=O(n)$. Dubroff, Gunby, Narayanan, and Spiro \cite{DGNS} obtained the following general lower bounds. A graph F is said to be \emph{2-balanced} if for every subgraph $F'$ of $F$, we have $\frac{|E(F')|-1}{|V(F')|-2} \leq \frac{|E(F)|-1}{|V(F)|-2}$. They showed that if $F$ is $2$-balanced and has at least $2$ edges, then $\mathrm{ex}(n, K_r, F) = \Omega(n^{2-\frac{|V(F)|-2}{e(F)}-1})$ for $2<r<|V(F)|$. 

For an integer $t \geq 1$, let $tF$ denote the vertex-disjoint union of $t$ copies of a graph $F$. Given two graphs $G$ and $H$ whose vertex sets are disjoint, we define the \emph{join} $G \vee H$ of $G$ and $H$ to be the graph obtained from $G \sqcup H$ (the vertex-disjoint union of $G$ and $H$) by adding all edges that have nonempty intersection with both $V(G)$ and $V(H)$. Gerbner, Methuku, and Vizer \cite{GMV19} investigated the interesting phenomenon where $\mathrm{ex}(n,H,tF)$ and $\mathrm{ex}(n,H,F)$ exhibit different asymptotic behaviors for various graphs $H$ and $F$. For graphs $F_1,\ldots,F_t$ (not necessarily distinct) with each $F_i \neq K_2$ and their vertex-disjoint union $F$, they proved that $$\mathrm{ex}(n, K_3,F) = \Theta\left(\max_{1\leq i \leq t}\{\mathrm{ex}(n,K_3,F_i)\} + \max_{1\leq i < j \leq t} \mathrm{ex}(n,\{F_i,F_j\})\right),$$ and notably, for any integer $t \geq 2$ and graph $H$, they have $$\mathrm{ex}(n,K_3,tH) = \Theta(\mathrm{ex}(n,K_3,H) + \mathrm{ex}(n,H)),$$ revealing that while the order of magnitude of $\mathrm{ex}(n,K_3,tF)$ may increase when $t$ changes from 1 to 2, no further increase occurs for $t \geq 2$. Hou, Yang and Zeng\cite{HYZ24} determined the exact value of $\mathrm{ex}(n, K_3, (t+1)C_{2k+1})$. Gao, Wu, and Xue \cite{GWX23} extended this study to generalized theta graphs $F = F(p_1,\ldots,p_k)$ consisting of internally vertex-disjoint paths of lengths $p_1,\ldots,p_k$ connecting two fixed vertices, investigating $\mathrm{ex}(n,K_3,tF)$ for such graphs with a color-critical edge when $3 \leq r \leq t+1$ and $n$ is sufficiently large. Gerbner, Methuku, and Vizer \cite{GMV19} offered the best bound about complete bipartite graph and even cycle until this paper by proving that for $t\geq 1, b \geq 2$,  $$\mathrm{ex}(n,K_3,tK_{2,b}) = (1+o(1))\left(\frac{(t-1)(b-1)^{1/2}}{2} + \frac{(b-1)^{3/2}}{6}\right)n^{3/2}, $$
and for any $r,k\geq 2$,
$$\mathrm{ex}(n,K_r,tC_{2k}) = O(n^{1+1/k}).$$

 Additional related results appear in \cite{GLLSY25, ZCG23,ZZZ25,ZFF24,GL12}, with Gerbner and Palmer \cite{GP25} providing an excellent survey of this area. Motivated by these results, we give some better bounds by counting cliques in  $tK_{a,b}$-free graphs, as well as in $tC_{2k}$-free graphs.

Fix an integer $r\ge 2$, an $r$-graph $\mathcal{H}$ is a collection of $r$-subsets of some finite set $V$. We identify a hypergraph $\mathcal{H}$ with its edge set and use $V(\mathcal{H})$ to denote its vertex set. For an $r$-graph $\mathcal{H}$ and a vertex $v\in V(\mathcal{H})$ the \emph{degree} $d_{\mathcal{H}}(v)$ of $v$ in $\mathcal{H}$ is the number of edges in $\mathcal{H}$ containing $v$. For $r=2$, let $G$ be a graph. For a subset $X$ of $V(G)$, we use $G[X]$ to denote the subgraph of $G$ induced by $X$. Denote $N_G(v)$ be the neighborhood of $v$ in $G$ and $d_G(v)=|N_G(v)|$. We use $k_r(G)$ to denote the number of cliques in $G$. We set for convenience that $\mathrm{ex}(n,K_0,H)=1$, $\mathrm{ex}(n,K_1,H)=n$ and $\mathrm{ex}(n,K_2,H)=\mathrm{ex}(n,H)$. Our first result extends the research of $\mathrm{ex}(n,K_3,tK_{2,b})$ to $\mathrm{ex}(n,K_r,tK_{a,b})$. And we get a sharp order of magnitude bound on it.

\begin{thm}\label{Kr(t+1)Kab}
    For any $t\geq r\geq 3$, $a\geq 2(r-1)$, $b\geq (a-1)!+1$, and $n$ large enough, we have that $$\mathrm{ex}(n,K_r,(t+1)K_{a,b})\leq \sum_{s=0}^{r}\binom{t}{r-s}\left(\frac{1}{s!}+o(1)\right)\left(b-1\right)^{\frac{s(s-1)}{2a}}(n-t)^{s-\frac{s(s-1)}{2a}},$$
    and 
    $$\mathrm{ex}(n,K_r,(t+1)K_{a,b})\geq \sum_{s=0}^{r}\binom{t}{r-s}\left(\frac{1}{s!} + o(1)\right)(n-t)^{s - \frac{s(s-1)}{2a}}.$$
\end{thm}
We note that the bounds established in Theorem~\ref{Kr(t+1)Kab} are asymptotically tight, with the gap between upper and lower bounds being at most a multiplicative constant depending on the parameters $a$, $b$, and $r$. Our second result gives more precise upper and lower bounds for $\mathrm{ex}(n,K_r,(t+1)C_{2k})$.   
\begin{thm}\label{main-theorem}
    For any $t\geq r \geq 3$, and any $(t+1)C_{2k}$-free graph $G$ on $n$ vertex where $n$ is large enough, we have that $$k_r(G)\leq \sum_{i=0}^{r}\binom{t}{r-i}\mathrm{ex}(n-t,K_i,C_{2k}). $$ 
    and $$k_r(G)\geq \binom{t}{r}+\binom{t}{r-1} \left(n-t\right)+\max_{2\leq s\leq r}\{\binom{t}{r-s} \mathrm{ex}(n-t,K_s,C_{2k})\}.$$
\end{thm}
In Section~\ref{Preliminaries and Lemmas} we show some Lemmas that will be used in the proof of the main theorems and in Section~\ref{Proof of Main Theorems} we give the proofs of Theorem~\ref{Kr(t+1)Kab} and Theorem~\ref{main-theorem}.


\section{Preliminaries and Lemmas}\label{Preliminaries and Lemmas}

The result of Alon and Shikhelman \cite{AS16} established upper and lower bounds for the number of cliques in $K_{a,b}$-free graphs. In particular, they determined the order of magnitude of $\mathrm{ex}(n,K_r,K_{a,b})$ when $r\geq 2$, $a\geq 2r-2$ and $b \geq (a-1)!+1$.
\begin{thm}[Alon and Shikhelman \cite{AS16}] \label{Kr Upper Bound in Kab-free}
    For any fixed $r \geq 2$ and $b \geq a \geq r-1$,$$ \mathrm{ex}\left(n, K_r, K_{a, b}\right) \leq\left(\frac{1}{r!}+o(1)\right)(b-1)^{\frac{r(r-1)}{2a}} n^{r-\frac{r(r-1)}{2a}}.$$
    Moreover, if $b\geq (a-1)!+1$, and $a\geq 2(r-1)$ then $$\mathrm{ex}\left(n, K_r, K_{a, b}\right)\geq \left(\frac{1}{r!}+o(1)\right)n^{r-\frac{r(r-1)}{2a}}.$$
\end{thm}
We remark that the lower bound construction in Theorem~\ref{Kr Upper Bound in Kab-free} comes from projective norm-graphs, denoted by $H(q,a)$, which were introduced in~\cite{ALON1999280}. Let $n = (1 + o(1))q^a$ and $d = (1 + o(1))q^{a-1}$. The graph $H(q,a)$ is $d$-regular and $K_{a,(a-1)!+1}$-free of $n$ vertices. Moreover, for any $2 \leq s \leq a/2 + 1$, it satisfies $$k_s(H(q,a)) \geq \left(\frac{1}{s!} + o(1)\right)n^{s - \frac{s(s-1)}{2a}}.$$ This immediately yields the following corollary:
\begin{cor}\label{Cor Lower Bound Kab-free}
    Let $r\geq 2$, $a\geq 2r-2$ and $b\geq (a-1)!+1$, then for each $n$ large enough, there exists a $K_{a,b}$-free $n$-vertex graph $G$ such that for each $1\leq s\leq r$, $$k_s(G)\geq\left(\frac{1}{s!}+o(1)\right)n^{s-\frac{s(s-1)}{2a}}.$$
\end{cor}

Let $P_{\ell}$ denote the path with $\ell$ vertices. The following lemma can be viewed as an extension of the celebrated Erd\H{o}s-Gallai theorem. It establishes that in graphs excluding $P_{\ell}$, the maximum number of cliques of each possible size grows linearly with $n$. This result provides strong control over the number of cliques contained in the neighborhood of each vertex in $C_{2k}$-free graphs. 
\begin{thm}[Luo \cite{L18}]\label{K_r is linear in Pt-free}
    Let $r,t\geq 2$ be positive integers. There exists a constant $C_{r,t}$ such that for each $n$, we have $\mathrm{ex}(n,K_r,P_t)\leq C_{r,t}n$. 
\end{thm}

Although we cannot determine the exact value of $\mathrm{ex}(n,C_{2k})$, by employing the Erd\H{o}s-R\'enyi first moment method, we can establish the following result, which is sufficient for us to prove Theorem~\ref{main-theorem}. We omit the proof of the lemma for simplicity.
\begin{lem}\label{Lower bound edge C2k-free}
    For each positive integer $k\geq 2$ there exists a constant $c_{k}$ such that for any $n$ large enough, there is a $C_{2k}$-free graph of order $n$ and size $m\geq c_kn^{1+\frac{1}{2k-1}}$.
\end{lem}

\section{Proof of Main Theorems}\label{Proof of Main Theorems}
We begin by proving Theorem~\ref{Kr(t+1)Kab}. The lower bound is obtained according to Corollary~\ref{Cor Lower Bound Kab-free}. The upper bound is proved by contradiction. Specifically, we assume that if a maximal $(t+1)K_{a,b}$-free graph contains more $K_r$ copies than our established upper bound, then such a graph must contain exactly $t$ vertices that appear in sufficiently many $K_r$ copies, thereby reaching a contradiction.

\begin{proof}[Proof of Theorem~\ref{Kr(t+1)Kab}]
    We first present the construction for the lower bound. Let $a \geq 2r-2$. By Corollary~\ref{Cor Lower Bound Kab-free}, there exists a $K_{a,b}$-free graph $H$ on $n-t$ vertices satisfying
$$
k_s(H) \geq \left(\frac{1}{s!} + o(1)\right)(n-t)^{s - \frac{s(s-1)}{2a}}
$$
for each $1\leq s\leq r$. Let $G = H \vee K_t$. Since $H$ is $K_{a,b}$-free, every copy of $K_{a,b}$ in $G$ must contain at least one vertex from $K_t$. Consequently, $G$ is $(t+1)K_{a,b}$-free. Moreover, $$\begin{aligned}
    k_r(G)&\geq \binom{t}{r}+\binom{t}{r-1}(n-t)+\sum_{s=2}^{r}\binom{t}{r-s}k_s(H)\\
    &\geq \sum_{s=0}^{r}\binom{t}{r-s}\left(\frac{1}{s!} + o(1)\right)(n-t)^{s - \frac{s(s-1)}{2a}}.
\end{aligned}$$
Now we start to prove the upper bound. We assume by contradiction that there is a $(t+1)K_{a,b}$-free graph $G$ of order $n$ with $$k_r(G)>\sum_{i=0}^{r}\binom{t}{r-i}\mathrm{ex}(n-t,K_i,K_{a,b}).$$ We also assume that $G$ has $t$ vertex-disjoint copies of $K_{a,b}$. Otherwise we can add as many edges as possible to $G$ while ensuring that $G$ contains no $t+1$ vertex-disjoint copies of $K_{a,b}$, until $G$ becomes a graph containing $t$ vertex-disjoint copies of $K_{a,b}$. Let $U\subseteq V(G)$ be the union of $(a+b)t$ vertices of the $t$ vertex-disjoint copies of $K_{a,b}$ in $G$ and $W=V(G)\setminus U$. Then $G[W]$ is $K_{a,b}$-free, since $G$ is $(t+1)K_{a,b}$-free. Define an auxiliary $r$-uniform hypergraph $\mathcal{H}$ as follows: $V(\mathcal{H})=V(G)$ and $x_1x_2\cdots x_r$ is an edge in $\mathcal{H}$ if and only if they induce a copy of $K_r$ in $G$. Define $$X=\{v\in V(\mathcal{H}):d_{\mathcal{H}}(v)\geq \frac{n^{\beta}}{2|U|}\},$$ where $$r-1-\frac{(r-1)(r-2)}{2(a-1)}<\beta<r-1-\frac{(r-1)(r-2)}{2a}$$ is determined by Theorem~\ref{Kr Upper Bound in Kab-free}. 
    Therefore, $$n^{\beta}<\mathrm{ex}(n,K_{r-1},K_{a,b})=\Theta\left(n^{r-1-\frac{(r-1)(r-2)}{2a}}\right)$$ when $n$ is sufficiently large.

    Next, we provide some characterizations of the set $X$.
    \begin{claim}\label{X is not empty Kab}
        $X\neq \emptyset$.
    \end{claim}
    \begin{proof}[Proof of Claim~\ref{X is not empty Kab}]
        Suppose that for every $v\in V(\mathcal{H})$, we have $d_{\mathcal{H}}(v)<\frac{n^{\beta}}{2|U|}$ and thus for $n$ large enough, we obtain that  $$\begin{aligned}
            k_r(G)&\leq \sum_{v\in U}d_{\mathcal{H}}(v)+k_r(G[W]) \\ &\leq |U|\cdot \frac{n^{\beta}}{2|U|}+\mathrm{ex}(n-|U|,K_r,K_{a,b})\\ &\leq \frac{1}{2}n^{\beta}+\mathrm{ex}(n-(a+b)t,K_{r},K_{a,b}) \\&< \mathrm{ex}(n-t,K_{r-1},K_{a,b})+\mathrm{ex}(n-t,K_{r},K_{a,b})\\ &<\sum_{i=0}^{r}\binom{t}{r-i}\mathrm{ex}(n-t,K_i,K_{a,b}),
        \end{aligned}$$
        which contradicts to our assumption.
    \end{proof}
    \begin{claim}\label{X small than t+1 Kab}
        $|X|\leq t$.
    \end{claim}
    \begin{proof}[Proof of Claim~\ref{X small than t+1 Kab}]
        Otherwise we assume that $|X|\geq t+1$ and $X'=\{x_1,\cdots, x_{t+1}\}\subseteq X$. We prove that the graph $G$ contains at least $t+1$ pairwise vertex-disjoint copies of $K_{a,b}$, and each $K_{a,b}$ passes through exactly one vertex from $X'$. Denote by $\mathcal{K}_r(x_1)$ and $\mathcal{K}_r(x_1,s)$ the set of $r$-cliques containing $x_1$, and the set of $r$-cliques containing $x_1$ with exactly $s$ vertices in $V(G)\setminus X'$, respectively. Notice that each $r$-clique in $\mathcal{K}_r(x_1,s)$ forms an $s$-clique in $G[N_G(v)\setminus X']$ and thus, $0\leq s\leq r-1$. If $G[N_{G}(x_1)\setminus X']$ is $K_{a-1,b}$-free, we obtain by Theorem~\ref{Kr Upper Bound in Kab-free} that there exists a constant $C_{s,a,b}$ such that 
        $$k_s(G[N_G(x_1)\setminus X'])\leq C_{s,a,b}d_G(v)^{s-\frac{s(s-1)}{2(a-1)}}\leq C_{s,a,b} n^{s-\frac{s(s-1)}{2(a-1)}}\leq C_{s,a,b}n^{r-1-\frac{(r-1)(r-2)}{2(a-1)}}.$$ 
        Moreover, by the definition of $X$, $x_i$ is contained in at least $\frac{n^{\beta}}{2|U|}$ copies of $K_r$ for each $i\in [t+1]$, which implies that
        $$\begin{aligned}
            \frac{n^{\beta}}{2|U|}&\leq d_{\mathcal{H}}(x_1)=|\mathcal{K}_r(x_1)| = \sum_{s=0}^{r-1}|\mathcal{K}_{r}(x_1,s)|\\ &< \sum_{s=0}^{r-1}\binom{|X'|}{r-s}k_s(G[N_G(x_1)\setminus X'])\\&\leq \sum_{s=0}^{r-1}C_{s,a,b}\binom{t+1}{r-s}n^{r-1-\frac{(r-1)(r-2)}{2(a-1)}}.
        \end{aligned}$$
        However, since $a,b,t$, and $\sum_{s=0}^{r-1} C_{s,a,b} \binom{t+1}{r-s}$ are all constants, the above inequality fails to hold according to our choice of $\beta$ when $n$ is sufficiently large. Consequently, there exists a $K_{a-1,b}$ in $N_G(x_1)\setminus X'$, which forms a $K_{a,b}$ together with $x_1$. Now we assume that we have found $i$ vertex-disjoint $K_{a,b}$ for $i\in[t]$, say $B_1,\cdots,B_i$, each passing through exactly one of $x_1,\ldots,x_i$ and avoiding all other vertices in $X'$. Let $X'_{i}=X'\cup\left(\cup_{j=1}^{i}V(B_i)\right)$ be the union of these vertices of $K_{a,b}$ and $X'$. We denote by $\mathcal{K}_r(x_{i+1})$ the collection of $K_r$ containing $x_{i+1}$, and by $\mathcal{K}_r(x_{i+1},s)$ those $K_r$ containing $x_{i+1}$ with exactly $s$ vertices in $N_{G}(x_{i+1})\setminus X'_{i}$. Consequently, we have $0\leq s\leq r-1$. If $G[N_{G}(x_{i+1})\setminus X'_{i}]$ is $K_{a-1,b}$-free, similar to the proof above, there exists constants $C'_{s,a,b}$  for $0\leq s\leq r-1$ such that $$\begin{aligned}
            \frac{n^{\beta}}{2|U|} &\leq d_{\mathcal{H}}(x_{i+1})=|\mathcal{K}_{r}(x_{i+1})|\leq \sum_{s=0}^{r-1}|\mathcal{K}_{r}(x_{i+1},s)|\\ &\leq \sum_{s=0}^{r-1} \binom{|X'_{i}|}{r-s} k_{s}\left(G[N_{G}(x_{i+1})\setminus X'_{i}]\right)\\&\leq \sum_{s=0}^{r-1} \binom{|X'_{i}|}{r-s} C'_{s,a,b} n^{s-\frac{s(s-1)}{2(a-1)}}\\ 
            &\leq \left(\sum_{s=0}^{r-1} \binom{|X'_{i}|}{r-s} C'_{s,a,b}\right) n^{r-1-\frac{(r-1)(r-2)}{2(a-1)}}.
        \end{aligned}$$
    The last inequality holds because $f(s)=s-\frac{s(s-1)}{2(a-1)}$ is non-decreasing when $0\leq s\leq r-1\leq a$. Again, since $a,b$ and $t$ are constants and $|X_{i}'| \leq (a+b)i+t \leq (a+b+1)t$, the inequality fails to hold when $n$ is sufficiently large. Therefore, $N_G(x_{i+1})\setminus X'_i$ contains a $K_{a-1,b}$, and thus we find a $K_{a,b}$ containing $x_{i+1}$ that is vertex-disjoint from all previously found copies of $K_{a,b}$ . Continuing this process, we can obtain $t+1$ vertex-disjoint copies of $K_{a,b}$, which contradicts the assumption that $G$ is $(t+1)K_{a,b}$-free.
    \end{proof}
    Now we define $X_1=X\cap U$.
    \begin{claim}\label{XU=t Kab}
    $|X_1|=t$.
    \end{claim}
    \begin{proof}[Proof of Claim~\ref{XU=t Kab}]
    We assume by contradiction that $|X_1|\leq t-1$. Recall that $|U| = (a+b)t$, $d_{\mathcal{H}}(u)\leq \frac{n^{\beta}}{2|U|}$ for each $u\in U\setminus X_1$, and $G[W]$ is $K_{a,b}$-free, we have that
    $$\begin{aligned}
        k_r(G)&\leq \sum_{u\in U\setminus X_1}d_{\mathcal{H}}(u)+\sum_{i=0}^{r}\binom{|X_1|}{r-i}\cdot k_{i}(G[W])\\ &\leq \frac{n^{\beta}}{2|U|}\cdot(|U|-|X_1|)+\sum_{i=0}^{r}\binom{|X_1|}{r-i}\cdot \mathrm{ex}(n-|U|,K_{i},K_{a,b})\\ &<\frac{n^{\beta}}{2}+\sum_{i=0}^{r}\binom{t-1}{r-i}\cdot \mathrm{ex}(n-t,K_{i},K_{a,b})\\ &\leq \frac{n^{\beta}}{2}+\sum_{i=0}^{r}\binom{t}{r-i}\cdot \mathrm{ex}(n-t,K_{i},K_{a,b})-\mathrm{ex}(n-t,K_{r-1}, K_{a,b})\\ &\leq \sum_{i=0}^{r}\binom{t}{r-i}\mathrm{ex}(n-t,K_i,K_{a,b})
    \end{aligned}$$
    holds for sufficiently large $n$, a contradiction.
    \end{proof}
    Let $X_2=V(G)\setminus X_1$.
    \begin{claim}\label{G[X2] is Kab free}
        $G[X_2]$ is $K_{a,b}$-free.
    \end{claim}
    \begin{proof}[{Proof of Claim~\ref{G[X2] is Kab free}}]
    If there exists a copy of $K_{a,b}$ denoted by $K$ in $G[X_2]$, by a similar proof of Claim~\ref{X small than t+1 Kab}, we can find $t$ copies of vertex disjoint $K_{a,b}$ in $G[V(G)\setminus V(K)]$, together with $K$, we will find $(t+1)K_{a,b}$ in $G$.
    \end{proof}
    Therefore, we can partite $V(G)=X_1\cup X_2$, where $|X_1|=t$ and $G[X_2]$ is $K_{a,b}$-free. Thus, $$k_r(G)\leq \sum_{i=0}^{r}\binom{t}{r-i}\mathrm{ex}(n-t,K_i,K_{a,b}),$$ 
    which contradicts to our previous assumption.
    Therefore, we have $$k_r(G)\leq \sum_{s=0}^{r}\binom{t}{r-s}\left(\frac{1}{s!}+o(1)\right)\left(b-1\right)^{\frac{s(s-1)}{2a}}(n-t)^{s-\frac{s(s-1)}{2a}}$$ according to Theorem~\ref{Kr Upper Bound in Kab-free}.
\end{proof}

Similar to Theorem~\ref{Kr(t+1)Kab}, we can derive Theorem~\ref{main-theorem}. Here we provide a sketch of the proof for Theorem~\ref{main-theorem}, where the key difference from the proof of Theorem~\ref{Kr(t+1)Kab} lies in the range selection for the high-degree vertices $d_\mathcal{H}(u)$. 
\begin{proof}[Sketch of proof of Theorem~\ref{main-theorem}]
   The lower bound can be constructed as follows. For each $2 \leq s \leq r$, let $H_s$ be a $C_{2k}$-free graph with $(n-t)$ vertices satisfying $k_s(H_s) = \mathrm{ex}(n-t,K_s,C_{2k})$. We select $H = H_i$ where the index $2 \leq i \leq r$ maximizes the expression:
$$
\binom{t}{r-i}\mathrm{ex}(n-t,K_i,C_{2k}) = \max_{2 \leq s \leq r} \left\{ \binom{t}{r-s} \mathrm{ex}(n-t,K_s,C_{2k}) \right\}.
$$
The graph $G = K_t \vee H$, which is evidently $(t+1)C_{2k}$-free, then contains the required number of $K_r$ copies for our purposes.
    
    Now we consider the upper bound. We assume by contradiction that there is a $(t+1)C_{2k}$-free graph $G$ of order $n$ with $$k_r(G)>\sum_{s=0}^{r}\binom{t}{r-s}\mathrm{ex}(n-t,K_s,C_{2k}).$$
    We also assume that $G$ has $t$ vertex disjoint $C_{2k}$ and let $U$ be the collection of $2kt$ vertices of these cycles. Define an $r$-uniform hypergraph $\mathcal{H}$ such that $V(\mathcal{H})=V(G)$ and $x_1x_2\cdots x_r$ is an edge in $\mathcal{H}$ if and only if they induce a copy of $K_r$ in G. Define $$X=\{v\in V(\mathcal{H}):d_{\mathcal{H}}(v)\geq \frac{n^{1+\epsilon}}{2|U|}\},$$ where $0<\epsilon<\frac{1}{2k-1}$ is determined by Theorem~\ref{K_r is linear in Pt-free} and Lemma~\ref{Lower bound edge C2k-free}.
    Set $X_1=X\cap U$ and $X_2=V(G)\setminus X_1$. We can prove that when $n$ is sufficiently large, we have $|X_1|=t$ and $G[X_2]$ is $C_{2k}$-free, which leads to a contradiction with our assumption.
\end{proof}

\section*{Acknowledge}
We gratefully acknowledge the Extremal Combinatorics and Probability Group at the Institute for Basic Science (IBS) for organizing the 3rd ECOPRO Student Research Program. In particular, Caihong Yang and Jiasheng Zeng would like to express their sincere appreciation to Professor Hong Liu for providing this exceptional research opportunity. This work was supported by the Institute for Basic Science (IBS-R029-C4), the National Key R\&D Program of China (Grant No. 2023YFA1010202), and the Central Guidance on Local Science and Technology Development Fund of Fujian Province (Grant No. 2023L3003), with the latter specifically supporting Caihong Yang's research.

\bibliographystyle{plain}
\bibliography{countinggraph}

\end{document}